\documentclass[10pt,fleqn]{article}
\usepackage{amssymb,amsmath,amsthm,mathrsfs,setspace}

\let\b=\beta

\let\d=\delta

\let\e=\varepsilon

\let\F=\Phi

\let\i=\iota

\let\r=\rho

\let\t=\tau

\let\vp=\varpi

\def\cR{{\cal R}}

\def\cV{{\cal V}}

\def\Lap{\Delta}

\def\so{\operatorname{SO}}

\def\sideremark#1{\ifvmode\leavevmode\fi\vadjust{\vbox to0pt{\vss
 \hbox to 0pt{\hskip\hsize\hskip1em
 \vbox{\hsize2cm\tiny\raggedright\pretolerance10000
 \noindent #1\hfill}\hss}\vbox to8pt{\vfil}\vss}}}%
\def\Lap{\triangle}

\newtheorem{theorem}{Theorem}[section]
\newtheorem{lemma}[theorem]{Lemma}
\newtheorem{remark}[theorem]{Remark}

\title{Intertwinors on Differential Forms over the Product of Spheres}
\author{Doojin Hong}
\date{\today}

\begin{document}

\maketitle

\begin{abstract}
We give explicit formulas for the intertwinors on the differential form bundles over $S^{p-1} \times S^{q-1}$ with the standard pseudo-Riemannian metric $g=-g_{{}_{S^{p-1}}}+g_{{}_{S^{q-1}}}$ of signature $(p-1,q-1)$. As a special case, we construct conformally invariant differential operators of all even orders. 
\end{abstract}
\section{Introduction}
Spectrum generating technique introduced by Branson, \'Olafsson, and {\O}rsted in \cite{BOO:96} has been successfully applied to multiplicity one cases and multiplicity two cases (\cite{BH:06}, \cite{BOO:96}, \cite{Hong:11}). 
In this paper, we apply the technique to differential form bundles over the product of spheres and present the spectral function for the intertwinors and construct explicitly conformally invariant differential operators of all even orders. This is a multiplicity two case and a generalization of Branson's work on $S^1\times S^{n-1}$, $n$ even in \cite{Branson:87}.
\section{Intertwinors on the differential form bundles}
Let $(\xi_{p-1},\cdots,\xi_0,\xi_1,\cdots, \xi_q)$ be homogeneous coordinates of $S^{p-1} \times S^{q-1}$. Set $\xi_0=\cos\tau$ and complete $\tau$ to a set of spherical angular coordinates on $S^{p-1}$. Likewise, set $\xi_1=\cos\rho$ and complete it to a shperical angular coordinates on 
$S^{q-1}$. Then we get a conformal vector field $T$ and its conformal factor $\vp$.
\begin{equation*}
T:=\cos\rho\sin \tau\partial_\tau+\cos \tau\sin\rho\partial_\rho \mbox{ and }\vp:=\cos \tau\cos\rho\, 
\end{equation*}
satisfying ${\mathcal L}_T=2\vp g$, where ${\mathcal L}_T$ is the Lie derivative with respect to $T$ and $g=-g_{{}_{S^{p-1}}}+g_{{}_{S^{q-1}}}$ is the standard pseudo-Riemannian metric on $S^{p-1} \times S^{q-1}$. Note that $w_0:=\cos\tau$ and $w_1:=\cos\rho$ are conformal factors corresponding to the conformal vector fields $Y_0:=\sin\tau\partial\tau$ and $Y_1:=\sin\rho\partial\rho$ on $S^{p-1}$ and $S^{q-1}$, respectively. 

Let $A$ be an intertwinor of order $2r$ on $k$-forms, that is, an operator satisfying the following relation (\cite{BH:07})
\begin{equation*}
A\left(\widetilde{\mathcal L}_T+\left(\frac{n}{2}-r\right)\vp\right)=\left(\widetilde{\mathcal L}_T-\left(\frac{n}{2}-r\right)\vp\right)A\, ,
\end{equation*}
where $\widetilde{\mathcal L}_T={\mathcal L}_T-k\vp$ is the reduced Lie derivative  (\cite{Kosmann:72}).

Note that for $\Psi \in \Lambda^k (S^{p-1} \times S^{q-1})$,
\begin{equation*}
({\mathcal L}_T-\nabla_T)\Psi=k\vp\Psi-\sin\r\sin \t(\e^0\i^1-\e^1\i^0)\Psi\,,
\end{equation*}
where $\e^0=\e(d\t), \e^1=\e(d\r)$, and $\i^0=\i(d\t), \i^1=\i(d\r)$ are exterior multiplications and interior multiplications, respectively. 

For $\F'\wedge \F\in\Lambda^a S^{p-1}\wedge \Lambda^{k-a} S^{q-1}$,
\begin{align*}
P(\F'\wedge \F):=&\sin\r\sin \t(\e^0\i^1-\e^1\i^0)(\F'\wedge \F)\\
=&(-1)^{a-1}\left\{[d',w_0]\F'\wedge [\d,w_1]\F+[\d',w_0]\F'\wedge [d,w_1]\F\right\}\, ,
\end{align*}
where $d'$ and $\d'$ (resp. $d$ and $\d$) are exterior derivative and coderivative
on $(S^{p-1},-g_{{}_{S^{p-1}}})$ (resp.$(S^{q-1},g_{{}_{S^{q-1}}})$) and $[,]$ is the ususal commutator relation.

Consider the standard {\it Riemannian} metric $g=g_{{}_{S^{p-1}}}+g_{{}_{S^{q-1}}}$ on $S^{p-1}\times S^{q-1}$ and {\it Riemannian} Bochner Laplacian 
$N:=-g^{\alpha\beta}\nabla_\alpha\nabla_\beta$. Then,  
\begin{equation}\label{BL}
[N,\varpi]=2\left(\nabla_T+\frac{n}2\varpi\right)
\end{equation}
on tensors of any type. Thus the intertwining relation becomes
\begin{equation}\label{int}
A\left(\frac{1}{2}[N,\vp]-P-r\vp\right)=\left(\frac{1}{2}[N,\vp]-P+r\vp\right)A\, .
\end{equation}

On $S^{q-1}$, the $a$-form $\Lap$-spectrum breaks up into a $\d d$-spectrum and a $d\d$-spectrum (where $\Lap=\d d+d\d$). Let $E_{a,d,j}$ (resp. $E_{a,\d,j}$) be the space of spherical harmonic $a$-forms of degree $j$ of the Hodge summand $\cR(d)$ (resp. $\cR(\d)$) on $(S^{q-1},g_{{}_{S^{q-1}}})$. And let $E'_{a,d',j'}$ and $E'_{a,\d',j'}$ be the corresponding objects on $(S^{p-1},g_{{}_{S^{p-1}}})$. 
Then the $k$-forms on $S^{p-1} \times S^{q-1}$ break up into irreducible $K=\so(p)\times \so(q)$-modules
\begin{align*}
&E'_{k-a,\d',j'}\wedge E_{a,\d,j}\mbox{ and } E'_{k-a,d',j'}\wedge E_{a,d,j},
\mbox{ multiplicity 1 types and}\\
&E'_{k-a,\d',j'}\wedge E_{a,d,j}\oplus E'_{k-a+1,d',j'}\wedge E_{a-1,\d,j}, 
\mbox{ multiplicity 2 type}.
\end{align*}

The conformal factor $\vp$ maps an irreducible $K$-module to a direct sum of irreducible $K$-modules by the selection rule (\cite{Branson:92}). We want to consider compressed intertwining relations on multiplicity 1 types and on multiplicity 2 types, respectively. 

$\bullet$ {\bf Multiplicity 1 type}\\
Given $E'_{k-a,\d',j'}\wedge E_{a,\d,j}$-type, consider projections as follows.
\begin{equation*}
\begin{array}{rcl}
E'_{k-a,\d',j'-1}\wedge E_{a,\d,j+1}&&E'_{k-a,\d',j'+1}\wedge E_{a,\d,j+1}\\
\nwarrow&&\nearrow\\
&E'_{k-a,\d',j'}\wedge E_{a,\d,j}&\\
\swarrow&&\searrow\\
E'_{k-a,\d',j'-1}\wedge E_{a,\d,j-1}&&E'_{k-a,\d',j'+1}\wedge E_{a,\d,j-1}
\end{array}
\end{equation*}

Note first that here $P\equiv 0$. Let $J'=j'+\dfrac{p-2}{2}$ and $J=j+\dfrac{q-2}{2}$. We consider the quotient of eigenvalues of the operator $A$. That is, the eigenvalue on one of the four target types over the eigenvalue on the source type in the above diagram.
Then, with respect to the above diagram, the transition quantities (the eigenvalue quotients) are
\begin{equation*}
\begin{array}{rcl}
\dfrac{-J'+J+1+r}{-J'+J+1-r}&&\dfrac{J'+J+1+r}{J'+J+1-r}\\
\nwarrow&&\nearrow\\
&\bullet&\\
\swarrow&&\searrow\\
\dfrac{-J'-J+1+r}{-J'-J+1-r}&&\dfrac{J'-J+1+r}{J'-J+1-r}
\end{array}
\end{equation*}
The above data can be organized in terms of the gamma function $\Gamma$ as follows.
\begin{equation}\label{m1}
\frac{\Gamma\left(\dfrac{J'+J+1+r}{2}\right)\Gamma\left(\dfrac{J'-J+1+r}{2}\right)}{
\Gamma\left(\dfrac{J'+J+1-r}{2}\right)\Gamma\left(\dfrac{J'-J+1-r}{2}\right)}\, .
\end{equation}
On the other multiplicity one type, $E'_{k-a,d',j'}\wedge E_{a,d,j}$, we have exactly the same transition quantities.

Projection from $E'_{k-a,\d',j'}\wedge E_{a,\d,j}$ to $E'_{k-a,d',j'}\wedge E_{a,d,j}$ gives the transition quantity 
\begin{equation}\label{s}
\dfrac{s-r}{s+r},\quad \mbox{where }s=(p+q-2-2k)/2.
\end{equation}
$\bullet$ {\bf Multiplicity 2 type}\\
Let $\cV(j',j)=E'_{k-a,\d',j'}\wedge E_{a,d,j}\oplus E'_{k-a+1,d',j'}\wedge E_{a-1,\d,j}$ and consider projections to the neighboring multiplicity 2 types.  
\begin{equation}\label{m2m2}
\begin{array}{rcl}
\cV(j'-1,j+1)&&\cV(j'+1,j+1)\\
\nwarrow&&\nearrow\\
&\cV(j',j)&\\
\swarrow&&\searrow\\
\cV(j'-1,j-1)&&\cV(j'+1,j-1)
\end{array}
\end{equation}

For $\d'\theta\wedge d\tau+d'\eta\wedge\d\zeta \in \cV(j',j)$, we view the operator $A$ as a $2\times 2$ matrix as follows.
\begin{equation*}
A(\d'\theta\wedge d\tau+d'\eta\wedge\d\zeta)
=A_{11}\d'\theta\wedge d\tau+A_{12}\d'd'\eta\wedge d\d\zeta+A_{21}\theta\wedge \tau+A_{22}d'\eta\wedge\d\zeta \, .
\end{equation*}
The following lemma shown by Branson in \cite{Branson:87} is crucial for the multiplicity 2 type case.
\begin{lemma}
\label{consts}
Let $w$ be a proper conformal vector field on $S^n$ with the standard Riemannian metric. Let $\varphi\in E_{k,\d,j}$ and $\psi\in E_{k,d,j}$. And let $w_{pr}^\pm$ be the projection of $w$ onto ${j\pm 1}$ types, respectively. Then,
\begin{equation*}\begin{array}{ll}
dw_{pr}^+\varphi=\dfrac{\mu+1}{\mu} w_{pr}^+d\varphi\, ,& 
dw_{pr}^-\varphi=\dfrac{\nu-1}{\nu} w_{pr}^-d\varphi\, ,\\
\d w_{pr}^+\psi=\dfrac{\beta+1}{\beta} w_{pr}^+\d\psi\, ,& 
\d w_{pr}^-\psi=\dfrac{\alpha-1}{\alpha} w_{pr}^+\d\psi\, ,
\end{array}
\end{equation*}
where
\begin{equation*}
\mu=j+k,\quad \nu=n-1-k+j,\quad \alpha=j-1+k,\quad \beta=n-k+j\, .
\end{equation*}
\qed
\end{lemma}

Let $\varpi_{pr}$ be the projection of $\varpi$ from $\cV(j',j)$ to one of its neighbors in the diagram (\ref{m2m2}) and $w_{0,pr}$ and $w_{1,pr}$ be the corresponding projections on $S^{p-1}$ and $S^{q-1}$. Then,
\begin{align*}
&\varpi_{pr}\d'\theta\wedge d\tau=w_{0,pr}\d'\theta\wedge w_{1,pr}d\tau=b\cdot\d'w_{0,pr}\theta\wedge dw_{1,pr}\tau ,\\
&\varpi_{pr}d'\eta\wedge\d\zeta=w_{0,pr}d'\eta\wedge w_{1,pr}\d\zeta, \text{ and}  \\ 
&\d'w_{0,pr}d'\eta\wedge d w_{1,pr}\d\zeta=b^{-1}\cdot\varpi_{pr}\d'd'\eta\wedge d\d\zeta,
\end{align*}
where $b$ is a constant to be determined by the lemma \ref{consts}. 

Thus,
\begin{equation*}
\begin{split}
\varpi_{pr} A\left[\begin{array}{c}\d'\theta\wedge d\tau\\d'\eta\wedge\d\zeta\end{array}\right]&=\left[\begin{array}{cc}A_{11}\,\varpi_{pr}\d'\theta\wedge d\tau & b\cdot A_{12}\cdot b^{-1}\,\varpi_{pr}\d'd'\eta\wedge d\d\zeta\\
b^{-1}\cdot A_{21}\cdot b\,\varpi_{pr}\theta\wedge\tau & A_{22}\,\varpi_{pr}d'\eta\wedge\d\zeta
\end{array}\right]\\
&=\left[\begin{array}{cc}A_{11} & b\cdot A_{12}\\b^{-1}\cdot A_{21} & A_{22}\end{array}\right]
\varpi_{pr}\left[\begin{array}{c}\d'\theta\wedge d\tau\\d'\eta\wedge\d\zeta\end{array}\right]\\
&=:A'\varpi_{pr}\left[\begin{array}{c}\d'\theta\wedge d\tau\\d'\eta\wedge\d\zeta\end{array}\right]\, .
\end{split}
\end{equation*}

Next, consider the projection $P_{pr}$ of the operator $P$ from $\cV(j',j)$ to one of its neighbors in the diagram (\ref{m2m2}). 
\begin{equation*}
\begin{split}
P_{pr}\d'\theta\wedge d\tau&=(-1)^{k-a-1}[d',w_{0,pr}]\d'\theta\wedge [\d,w_{1,pr}]d\tau \\
&=(-1)^{k-a-1}\{d'(w_{0,pr}\d'\theta)-w_{0,pr}d'\d'\theta\}\wedge \{\d(w_{1,pr}d\tau)-w_{1,pr}\d d\tau\}\\
&=(-1)^{k-a-1}(\lambda_1-1)w_{0,pr}d'\d'\theta\wedge(\lambda_2-1)w_{1,pr}\d d\tau \\
&=(-1)^{k-a-1}(\lambda_1-1)l_1w_{0,pr}\theta\wedge(\lambda_2-1)l_2w_{1,pr}\tau\, ,
\end{split}
\end{equation*}
where constants $\lambda_1$, $\lambda_2$, $r_1$, $r_2$ are determined by the lemma~\ref{consts} and $l_1=-\alpha\beta$ at $E'_{k-a+1,d',j'}$ on $S^{p-1}$ and $l_2=\mu\nu$ at $E_{a-1,\d,j}$ on $S^{q-1}$ with the standard Riemannian metrics on both spaces. 

On the other hand,
\begin{equation*}
\begin{split}
P_{pr}d'\eta\wedge\d\zeta&=(-1)^{k-a}[\d',w_{0,pr}]d'\eta\wedge [d,w_{1,pr}]\d\zeta \\
&=(-1)^{k-a}\{\d'(w_{0,pr}d'\eta)-w_{0,pr}\d'd'\eta\}\wedge\{d(w_{1,pr}\d\zeta)-w_{1,pr}d\d\zeta\}\\
&=(-1)^{k-a}(1-r_1)\d'(w_{0,pr}d'\eta)\wedge(1-r_2)d(w_{1,pr}\d\zeta) \, .
\end{split}
\end{equation*}

Since $w_{0,pr}\d'\theta\wedge w_{1,pr}d\tau=b\cdot\d'w_{0,pr}\theta\wedge dw_{1,pr}\tau$, $P_{pr}\left[\begin{array}{c}\d'\theta\wedge d\tau\\d'\eta\wedge\d\zeta\end{array}\right]$ equals $(-1)^{k-a-1}$ times
\begin{equation*}
\left[\begin{array}{cc}0 & -(1-r_1)(1-r_2)\\(\lambda_1-1)(\lambda_2-1)l_1l_2b^{-1} & 0\end{array}\right]
\varpi_{pr}\left[\begin{array}{c}\d'\theta\wedge d\tau\\d'\eta\wedge\d\zeta\end{array}\right] .
\end{equation*}

Then, the compressed intertwining relation between $\cV(j',j)$ and one of the neighbors in the diagram (\ref{m2m2}) becomes
\begin{equation*}
AB^-\varpi_{pr}\left[\begin{array}{c}\d'\theta\wedge d\tau\\d'\eta\wedge\d\zeta\end{array}\right]=B^+A'\varpi_{pr}\left[\begin{array}{c}\d'\theta\wedge d\tau\\d'\eta\wedge\d\zeta\end{array}\right]\, ,
\end{equation*}
where
\begin{equation*}
\begin{split}
&B^-:=\left[\begin{array}{cc}\frac{1}{2}N_c-r & (-1)^{k-a-1}(1-r_1)(1-r_2)\\
(-1)^{k-a}(\lambda_1-1)(\lambda_2-1)l_1l_2b^{-1} & \frac{1}{2}N_c-r \end{array}\right] ,\\
&B^+:=\left[\begin{array}{cc}\frac{1}{2}N_c+r & (-1)^{k-a-1}(1-r_1)(1-r_2)\\
(-1)^{k-a}(\lambda_1-1)(\lambda_2-1)l_1l_2b^{-1} & \frac{1}{2}N_c+r \end{array}\right],\, \text{ and }\\
&N_c=N \text{ in }(\ref{BL})\text{ at the target space}-N \text{ at }\cV(j',j). 
\end{split}
\end{equation*}
Therefore,
\begin{equation*}
\frac{\text{det} A\text{ at a neighbor}}{\text{det} A\text{ at }\cV(j',j)}=\frac{\text{det} A}{\text{det} A'}=\frac{\text{det} B^+}{\text{det} B^-}\, .
\end{equation*}

The transition quantities are 
\begin{equation*}
\begin{array}{rcl}
\dfrac{(-J'+J+r)(-J'+J+2+r)}{(-J'+J-r)(-J'+J+2-r)}&&\dfrac{(J'+J+r)(J'+J+2+r)}{(J'+J-r)(J'+J+2-r)}\\
\nwarrow&&\nearrow\\
&\bullet&\\
\swarrow&&\searrow\\
\dfrac{(-J'-J+r)(-J'-J+2+r)}{(-J'-J-r)(-J'-J+2-r)}&&\dfrac{(J'-J+r)(J'-J+2+r)}{(J'-J-r)(J'-J+2-r)}
\end{array}
\end{equation*}
The corresponding gamma function expression is
\begin{equation}\label{m2}
\frac{\Gamma\left(\dfrac{J'+J+r}{2}\right)\Gamma\left(\dfrac{J'+J+2+r}{2}\right)\Gamma\left(\dfrac{J'-J+r}{2}\right)\Gamma\left(\dfrac{J'-J+2+r}{2}\right)}
{\Gamma\left(\dfrac{J'+J-r}{2}\right)\Gamma\left(\dfrac{J'+J+2-r}{2}\right)\Gamma\left(\dfrac{J'-J-r}{2}\right)\Gamma\left(\dfrac{J'-J+2-r}{2}\right)}\, .
\end{equation}
$\bullet$ {\bf Interface between multiplicity 1 and 2 types}\\
Now we want to consider a relationship between multiplicity one type and multiplicity two type. Let $\d'\theta\wedge\d\tau\in E'_{k-a+1,\d',j'}\wedge E_{a-1,\d,j+1}$. The projection of the right hand side of the intertwining relation 
(\ref{int}) onto $E'_{k-a,\d',j'}\wedge E_{a,d,j}$ is 
\begin{equation*}
(-1)^{k-a+1}(1-c_1)t_1\cdot \d' w_{0,pr}\d'\theta\wedge d w_{1,pr}\d\tau\, ,
\end{equation*}
where
\begin{equation*}
\begin{split}
&w_{0,pr} \text{ is the projection to }E'_{k-a+1,d',j'},\, 
w_{1,pr} \text{ is the projection to }E_{a-1,\d,j},\\
&t_1 \text{ is the eigenvalue of }A\text{ on }E'_{k-a+1,\d',j'}\wedge E_{a-1,\d,j+1},\, \text{ and }\\
&w_{1,pr}d\d\tau=c_1dw_{1,pr}\d\tau \text{ for some constant }c_1\text{ by the Lemma \ref{consts}}\, .
\end{split}
\end{equation*}
On the other hand, the projection of the left hand side of 
(\ref{int}) is
\begin{equation*}
(-1)^{k-a+1}(1-c_1)A_{11}+\left(\frac{1}{2}N_1-r\right)A_{12}\cdot \d' w_{0,pr}\d'\theta\wedge d w_{1,pr}\d\tau\, ,
\end{equation*}
where
\begin{equation*}
\begin{split}
&N_1=N \text{ on }E'_{k-a+1,d',j'}\wedge E_{a-1,\d,j}
-N \text{ on }E'_{k-a+1,\d',j'}\wedge E_{a-1,\d,j+1},\\
&A_{11},A_{12} \text{ are entries of }A\text{ on }
E'_{k-a,\d',j'}\wedge E_{a,d,j}\oplus E'_{k-a+1,d',j'}\wedge E_{a-1,\d,j}\, .
\end{split}
\end{equation*}
So we have
\begin{equation}
(-1)^{k-a+1}(1-c_1)A_{11}+\left(\frac{1}{2}N_1-r\right)A_{12}
=(-1)^{k-a+1}(1-c_1)t_1\, .
\end{equation}
The projection from $E'_{k-a+1,\d',j'}\wedge E_{a-1,\d,j+1}$ onto $E'_{k-a+1,d',j'}\wedge E_{a-1,\d,j}$ yields
\begin{equation}
(-1)^{k-a+1}(1-c_1)A_{21}+\left(\frac{1}{2}N_1-r\right)A_{22}=\left(\frac{1}{2}N_1+r\right)t_1\, .
\end{equation}
Let $d'\eta\wedge d\zeta\in E'_{k-a,d',j'}\wedge E_{a,d,j+1}$. Then the projections of $d'\eta\wedge d\zeta$ onto $E'_{k-a+1,d',j'}\wedge E_{a-1,\d,j}$ and $E'_{k-a,\d',j'}\wedge E_{a,d,j}$, respectively yield
\begin{align}
&\left(\frac{1}{2}N_2-r\right)\frac{1}{m_1m_2}A_{21}+(-1)^{k-a}(1-c_2)A_{22}
=(-1)^{k-a}(1-c_2)t_2,\\
&\left(\frac{1}{2}N_2-r\right)A_{11}+(-1)^{k-a}(1-c_2)m_1m_2A_{12}
=\left(\frac{1}{2}N_2+r\right)t_2,
\end{align}
where
\begin{equation*}
\begin{split}
&N_2=N \text{ on }E'_{k-a,\d',j'}\wedge E_{a,d,j}
-N \text{ on }E'_{k-a,d',j'}\wedge E_{a,d,j+1},\\
&A_{21},A_{22} \text{ are entries of }A\text{ on }
E'_{k-a,\d',j'}\wedge E_{a,d,j}\oplus E'_{k-a+1,d',j'}\wedge E_{a-1,\d,j},\\
&t_2 \text{ is the eigenvalue of }A\text{ on }E'_{k-a,d',j'}\wedge E_{a,d,j+1},\\
&m_1 \text{ is }(-1)\cdot \d'd' \text{ eigenvalue on }E'_{k-a,\d',j'} \text{ in the Riemannian }S^{p-1},\\
&m_2 \text{ is }d\d \text{ eigenvalue on }E_{a,d,j} \text{ in }S^{q-1},\text{ and}\\
&w_{1,pr}\d d\zeta=c_2\d w_{1,pr}d\zeta. \text{ for some constant }c_2\text{ by the Lemma \ref{consts}}. \\
&\text{Here }w_{1,pr} \text{ is the projection to }E_{a,d,j}.
\end{split}
\end{equation*}
Between $E'_{k-a,\d',j'}\wedge E_{a,\d,j+1}$ and $E'_{k-a+1,\d',j'}\wedge E_{a-1,\d,j+1}$ the compressed intertwining relation shows that
\begin{multline*} 
\text{eigenvalue of }A\text{ on }E'_{k-a,\d',j'}\wedge E_{a,\d,j+1}\\
=\quad\text{eigenvalue of }A\text{ on }E'_{k-a+1,\d',j'}\wedge E_{a-1,\d,j+1}\, .
\end{multline*}
Thus, $\displaystyle{t_2=\frac{s-r}{s+r}\cdot t_1}$ by (\ref{s}) and we can write $A_{11}$, 
$A_{12}$, $A_{21}$, and $A_{22}$ in terms of $t_1$ using equations (3) through (6) as follows. 
\begin{equation}\label{m2-1}
\begin{array}{l}
A_{11}:\\
t\cdot\big((s+r)m_1+(s-r)m_2+(s+r)(s-r)\{(q-p)/2+k-2a+1-r\}\big),\\
A_{12}:\\
t\cdot(-1)^{k-a+1}\cdot 2r,\\
A_{21}:\\
t\cdot(-1)^{k-a+1}\cdot 2r\cdot m_1m_2,\\
A_{22}:\\
t\cdot\big((s-r)m_1+(s+r)m_2+(s+r)(s-r)\{(q-p)/2+k-2a+1+r\}\big),
\end{array}
\end{equation}
where 
\begin{equation*}
t=-\frac{t_1}{(J'+J+r)(J'-J-r)(s+r)}. 
\end{equation*}
In particular, we get 
\begin{equation*}
\text{det }A=\frac{J'+J-r}{J'+J+r}\cdot \frac{J'-J+r}{J'-J-r}\cdot \frac{s-r}{s+r}\cdot t_1^2\, .
\end{equation*}
From (\ref{m1}),(\ref{m2}), we get
\begin{equation*}
t_1^2=\frac{s+r}{s-r}\cdot
\left(\frac{\Gamma\left(\dfrac{J'+J+2+r}{2}\right)\Gamma\left(\dfrac{J'-J+r}{2}\right)}{
\Gamma\left(\dfrac{J'+J+2-r}{2}\right)\Gamma\left(\dfrac{J'-J-r}{2}\right)}\right)^2\, .
\end{equation*}
Therefore, we normalize the operator $A$ as follows.
\begin{theorem}\label{thm1}
The operator $A$ is 
\begin{align*}
\sqrt{\frac{s-r}{s+r}}\cdot\frac{\Gamma\left(\dfrac{J'+J+1+r}{2}\right)\Gamma\left(\dfrac{J'-J+1+r}{2}\right)}{
\Gamma\left(\dfrac{J'+J+1-r}{2}\right)\Gamma\left(\dfrac{J'-J+1-r}{2}\right)}
\text{ on }E_{k-a,d',j'}\wedge E_{a,d,j},\\
\sqrt{\frac{s+r}{s-r}}\cdot\frac{\Gamma\left(\dfrac{J'+J+1+r}{2}\right)\Gamma\left(\dfrac{J'-J+1+r}{2}\right)}{
\Gamma\left(\dfrac{J'+J+1-r}{2}\right)\Gamma\left(\dfrac{J'-J+1-r}{2}\right)}
\text{ on }E_{k-a,\d',j'}\wedge E_{a,\d,j}\, .
\end{align*}
And det $A$ is, on $E'_{k-a,\d',j'}\wedge E_{a,d,j}\oplus E'_{k-a+1,d',j'}\wedge E_{a-1,\d,j}$,
\begin{equation*}
\frac{\Gamma\left(\dfrac{J'+J+r}{2}\right)\Gamma\left(\dfrac{J'+J+2+r}{2}\right)\Gamma\left(\dfrac{J'-J+r}{2}\right)\Gamma\left(\dfrac{J'-J+2+r}{2}\right)}
{\Gamma\left(\dfrac{J'+J-r}{2}\right)\Gamma\left(\dfrac{J'+J+2-r}{2}\right)\Gamma\left(\dfrac{J'-J-r}{2}\right)\Gamma\left(\dfrac{J'-J+2-r}{2}\right)}\, .
\end{equation*}
\qed
\end{theorem}
\section{Conformally invariant differential operators of even orders}
Let $R$ be the scalar curvature and $R_{\alpha\beta}$ the Ricci tensor of a pseudo-Riemannian manifold $(M,g)$ of dimension $n\ge 3$. Consider the conformally covariant differential operator of order 2 on $k$-forms (\cite{Branson:87})
\begin{equation*}
D_{2,k}=(s+1)\d d+(s-1)d\d+(s+1)(s-1)(\widetilde{R}-2V\#),
\end{equation*}
where
\begin{align*}
&s=(n-2k)/2,\quad \widetilde{R}=R/2(n-1),\quad V_{\alpha\beta}=(R_{\alpha\beta}
-\widetilde{R}g_{\alpha\beta})/(n-2),\text{ and}\\
&(V\#\varphi)_{a_1 a_2\cdots a_k}=\sum_{i=1}^{k} V^\alpha{}_{a_i}\varphi_{a_1\cdots a_{i-1}\alpha a_{i+1}\cdots a_k} \text{ for }\varphi\in \Lambda^k M . 
\end{align*}

In the case of $(S^{p-1} \times S^{q-1}, -g_{{}_{S^{p-1}}}+g_{{}_{S^{q-1}}})$, 
\begin{align*}
&R_{\alpha\beta}=\left[\begin{array}{cc}(p-2)g_{{}_{S^{p-1}}}&0\\
0&(q-2)g_{{}_{S^{q-1}}}\end{array}\right],\\
&R=-(p-1)(p-2)+(q-1)(q-2),\text{ and}\\
&\widetilde{R}=\frac{q-p}{2},\quad \text{and} \quad V_{\alpha\beta}=\frac{1}{2}\left[\begin{array}{cc}g_{{}_{S^{p-1}}}&0\\
0&g_{{}_{S^{q-1}}}\end{array}\right] .
\end{align*}
So for $\Phi'\wedge\Phi\in \Lambda^{k-1}S^{p-1}\wedge\Lambda^a S^{q-1}$,
\begin{equation*}
V\#(\Phi'\wedge\Phi)=\frac{1}{2}\{-(k-a)+a\}\cdot\Phi'\wedge\Phi .
\end{equation*}
$D_{2,k}$ is
\begin{align*}
&(s-1)(d'\d'\Phi'\wedge\Phi+\Phi'\wedge d\d \Phi)+(s+1)(s-1)\left(\frac{q-p}{2}+k-2a\right)\Phi'\wedge\Phi\\
&\qquad\text{on}\quad E'_{k-a,d',j'}\wedge E_{a,d,j},\\
&(s+1)(\d'd'\Phi'\wedge\Phi+\Phi'\wedge \d d \Phi)+(s+1)(s-1)\left(\frac{q-p}{2}+k-2a\right)\Phi'\wedge\Phi\\
&\qquad\text{on}\quad E'_{k-a,\d',j'}\wedge E_{a,\d,j}.
\end{align*}
So $D_{2,k}$ acts by a constant
\begin{align*}
&(s-1)(J+J')(J-J')\quad \text{on}\quad E'_{k-a,d',j'}\wedge E_{a,d,j}\quad \text{and}\\
&(s+1)(J+J')(J-J')\quad \text{on}\quad E'_{k-a,\d',j'}\wedge E_{a,\d,j}.
\end{align*}
On $\left[\begin{array}{c}E'_{k-a,\d',j'}\wedge E_{a,d,j}\\ E'_{k-a+1,d',j'}\wedge E_{a-1,\d,j}\end{array}\right]$, $D_{2,k}$ can be viewed as a $2\times 2$ matrix whose entries are
\begin{align*}
&(1,1):\, (s+1)\d'd'+(s-1)d\d+(s+1)(s-1)\{(q-p)/2+k-2a\}, \\
&(1,2):\, (-1)^{k-a+1}\d'd, \\
&(2,1):\, (-1)^{k-a+1}d'\d, \\
&(2,2):\, (s-1)d'\d'+(s+1)\d d+(s+1)(s-1)\{(q-p)/2+k-2(a-1)\} .
\end{align*}
Thus $D_{2,k}=\sqrt{(s+1)(s-1)}\cdot A$, where $A$ is the intertwinor in the Theorem \ref{thm1} with $r=1$.

Now we want to construct differential intertwinors of all even orders. Let us consider the following operators.
\begin{equation*}
\begin{array}{ll}

A'=\sqrt{-\d'd'+\{(p-2)/2-(k-a)\}^2},& A=\sqrt{\d d+\{(q-2)/2-a\}^2}\\
B'=\sqrt{-d'\d'+\{(p-2)/2-(k-a)\}^2},& B=\sqrt{d\d+\{(q-2)/2-(a-2)\}^2}.
\end{array}
\end{equation*}
$A'$ (resp. $B'$) acts by the constant $J'$ on $E'_{k-a,\d',j'}$ (resp. $E'_{k-a+1,d',j'}$) and $A$ (resp. $B$) acts by the constant $J$ on $E_{a,\d,j}$ (resp. $E_{a-1,d,j}$). And let
\begin{align*}
&C'=\begin{cases}\sqrt{-\d'd'+\{(p-2)/2-(k-a)\}^2}&\mbox{ on }E'_{k-a,\d',j'},\\
\sqrt{-d'\d'+\{(p-2)/2-(k-a)\}^2}&\mbox{ on }E'_{k-a+1,d',j'},\end{cases} \\
&C=\begin{cases}\sqrt{d\d+\{(q-2)/2-(a-1)\}^2} &\mbox{ on }E_{a,d,j},\\
\sqrt{\d d+\{(q-2)/2-(a-1)\}^2} &\mbox{ on }E_{a-1,\d,j}.\end{cases}
\end{align*}
Then $C'$ (resp. $C$) acts by the constant $J'$ (resp. $J$) on the types they are defined.
\begin{theorem}
Let $\mathscr{B}$ be the $2\times 2$ matrix on $\left[\begin{array}{c}E'_{k-a,\d',j'}\wedge E_{a,d,j}\\ E'_{k-a+1,d',j'}\wedge E_{a-1,\d,j}\end{array}\right]$ defined as follows.
\begin{equation*}
\begin{split}
&(1,1):\quad (s+r)\d'd'+(s-r)d\d+(s+r)(s-r)\{(q-p)/2+k-2a+1-r\}, \\
&(1,2):\quad (-1)^{k-a+1}\cdot 2r\cdot\d'd, \\
&(2,1):\quad (-1)^{k-a+1}\cdot 2r\cdot d'\d, \\
&(2,2):\quad (s-r)d'\d'+(s+r)\d d+(s+r)(s-r)\{(q-p)/2+k-2a+1+r\} .
\end{split}
\end{equation*}
Then the operator $D_{2r,k}$ defined by,
\begin{align*}
&\text{for }\underline{r\text{ odd,}}\\
&(s+r)(A'+A)(A'-A)\prod_{l=1}^{(r-1)/2}(A'+A+2l)(A'+A-2l)\\
&\quad\cdot (A'-A+2l)(A'-A-2l)\quad \text{on }E'_{k-a,\d',j'}\wedge E_{a,\d,j},\\
&(s-r)(B'+B)(B'-B)\prod_{l=1}^{(r-1)/2}(B'+B+2l)(B'+B-2l)\\
&\quad\cdot (B'-B+2l)(B'-B-2l)\quad \text{on }E'_{k-a+1,d',j'}\wedge E_{a-1,d,j},\\
&(-1)\Bigg\{\prod_{l=1}^{(r-1)/2}(C'+C+(2l-1))(C'+C-(2l-1))(C'-C+(2l-1))\\
&\cdot (C'-C-(2l-1))\Bigg\}\cdot\mathscr{B}
\quad \text{on }\left[\begin{array}{l}E'_{k-a,\d',j'}\wedge E_{a,d,j}\\
E'_{k-a+1,d',j'}\wedge E_{a-1,\d,j}\end{array}\right] 
\end{align*}
\begin{align*}
&\text{and }\text{ for }\underline{r\text{ even,}}\\
&(s+r)\prod_{l=1}^{r/2}(A'+A+(2l-1))(A'+A-(2l-1))\\
&\quad\cdot (A'-A+(2l-1))(A'-A-(2l-1))\quad \text{on }E'_{k-a,\d',j'}\wedge E_{a,\d,j},\\
&(s-r)\prod_{l=1}^{r/2}(B'+B+(2l-1))(B'+B-(2l-1))\\
&\quad\cdot (B'-B+(2l-1))(B'-B-(2l-1))\quad \text{on }E'_{k-a+1,d',j'}\wedge E_{a-1,d,j},\\
&(-1)(C'+C)(C'-C)\Bigg\{\prod_{l=1}^{(r-2)/2}(C'+C+2l)(C'+C-2l)(C'-C+2l)\\
&\cdot (C'-C-2l)\Bigg\}\cdot\mathscr{B}
\quad \text{on }\left[\begin{array}{l}E'_{k-a,\d',j'}\wedge E_{a,d,j}\\
E'_{k-a+1,d',j'}\wedge E_{a-1,\d,j}\end{array}\right] 
\end{align*}
is a conformally covariant differential operator with leading term $(-1)^r$ times
\begin{equation*}
(s+r)(\d^{(n)}d^{(n)})^r+(s-r)(d^{(n)}\d^{(n)})^r ,
\end{equation*}
where $d^{(n)}$ is the exterior derivative and $\d^{(n)}$ is the coderivative in $S^{p-1} \times S^{q-1}$ with metric signature $(p-1,q-1)$.
\end{theorem}
\begin{proof}
Conformal covariance is clear from (\ref{m2-1}) and Theorem \ref{thm1}. Now, we shall show that $D_{2r,k}=D_1+D_2-D_3$ for some differential operators $D_1$, $D_2$, and $D_3$. On 
$\left[\begin{array}{c}
E'_{k-a,\d',j'}\wedge E_{a,\d,j}\oplus E'_{k-a,\d',j'}\wedge E_{a,d,j}\\
E'_{k-a+1,d',j'}\wedge E_{a-1,\d,j}\oplus E'_{k-a+1,d',j'}\wedge E_{a-1,d,j}
\end{array}\right]$, 
let $D_1$ be the $2\times 2$ matrix defined by 
\begin{align*}
&\text{on (1,1) : }\\
&(s+r)(-\d'd'-\d d+h_0-h_1)\\
&\cdot\prod_{l=1}^{(r-1)/2}\Big((-\d'd'+h_0)^2-2(\d d+h_1+(2l)^2)(-\d'd'+h_0)\\
&\qquad+(\d d+h_1-(2l)^2)^2\Big)\quad\text{for }\underline{r \text{ odd}},\\
&(s+r)\prod_{l=1}^{r/2}\Big((-\d'd'+h_0)^2-2(\d d+h_1+(2l-1)^2)(-\d'd'+h_0)\\
&\qquad +(\d d+h_1-(2l-1)^2)^2\Big)\quad\text{for }\underline{r \text{ even}},\\
&\text{where }h_0=((p-2)/2-(k-a))^2\text{ and } h_1=((q-2)/2-a)^2,
\end{align*}
\begin{align*}
&\text{on (2,2) : }\\
&(s-r)(-d'\d'-d \d+h_0-h_2)\\
&\cdot\prod_{l=1}^{(r-1)/2}\Big((-d'\d'+h_0)^2-2(d \d+h_2+(2l)^2)(-d'\d'+h_0)\\
&\qquad +(d \d+h_2-(2l)^2)^2\Big)\quad\text{for }\underline{r \text{ odd}},\\
&(s-r)\prod_{l=1}^{r/2}\Big((-d'\d'+h_0)^2-2(d \d+h_2+(2l-1)^2)(-d'\d'+h_0)\\
&\qquad +(d \d+h_2-(2l-1)^2)^2\Big)\quad\text{for }\underline{r \text{ even}},\\
&\text{where }h_2=((q-2)/2-(a-2))^2 . 
\end{align*}
To define $D_2$, we define $\mathscr{C}$ first to be 0 on (1,2), (2,1), and by 
\begin{align*}
&\text{on (1,1) : }\\
&(-1)\prod_{l=1}^{(r-1)/2}\Big((-\d'd'+h_0)^2-2(d\d+h_3+(2l-1)^2)(-\d'd'+h_0)\\
&\qquad +(d\d+h_3-(2l-1)^2)^2\Big)\quad\text{for }\underline{r \text{ odd}},\\
&(-1)(-\d'd'-d\d+h_0-h_3)\\
&\cdot\prod_{l=1}^{(r-2)/2}\Big((-\d'd'+h_0)^2-2(d\d+h_3+(2l)^2)(-\d'd'+h_0)\\
&\qquad +(d\d+h_3-(2l)^2)^2\Big)\quad\text{for }\underline{r \text{ even}},
\end{align*}
\begin{align*}
&\text{on (2,2) : }\\
&(-1)\prod_{l=1}^{(r-1)/2}\Big((-d'\d'+h_0)^2-2(\d d+h_3+(2l-1)^2)(-d'\d'+h_0)\\
&\qquad+(\d d+h_3-(2l-1)^2)^2\Big)\quad\text{for }\underline{r \text{ odd}},\\
&(-1)(-d'\d'-\d d+h_0-h_3)\\
&\cdot\prod_{l=1}^{(r-2)/2}\Big((-d'\d'+h_0)^2-2(\d d+h_3+(2l)^2)(-d'\d'+h_0)\\
&\qquad+(\d d+h_3-(2l)^2)^2\Big)\quad\text{for }\underline{r \text{ even}},\\
&\text{where }h_3=((q-2)/2-(a-1))^2.
\end{align*}
Now define $D_2:=\mathscr{C}\cdot\mathscr{B}$.
Clearly, 
\begin{align*}
&D_1=D_{2r,k} \text{ on }E'_{k-a,\d',j'}\wedge E_{a,\d,j}\oplus E'_{k-a+1,d',j'}\wedge E_{a-1,d,j}\text{ and}\\
&D_2=D_{2r,k} \text{ on }E'_{k-a,\d',j'}\wedge E_{a,d,j}\oplus E'_{k-a+1,d',j'}\wedge E_{a-1,\d,j} .
\end{align*}
Note that on $E'_{k-a,\d',j'}\wedge E_{a,d,j}$, $D_1$ acts as the polynomial $\mathscr{P}_1$ in $\d'd'$ of order $r$ with leading coefficient $(-1)^r(s+r)$ and roots
\begin{align*}
&h_0-(h_1\pm2l)^2=h_0-((q-2)/2-a\pm2l)^2,\\
&\qquad l=0,1,\dots,(r-1)/2 \quad \text{for }r\text{ odd},\\
&h_0-((q-2)/2-a\pm(2l-1))^2,\quad l=1,2,\dots,r/2 \quad \text{for }r\text{ even}.
\end{align*}
On $E'_{k-a+1,d',j'}\wedge E_{a-1,\d,j}$,on the other hand, $D_1$ acts as the polynomial $\mathscr{P}_2$ in $d'\d'$ of order $r$ with leading coefficient $(-1)^r(s-r)$ and roots
\begin{align*}
&h_0-((q-2)/2-(a-2)\pm2l)^2,\quad l=0,1,\dots,(r-1)/2 \quad \text{for }r\text{ odd},\\
&h_0-((q-2)/2-(a-2)\pm(2l-1))^2,\quad l=1,2,\dots,r/2 \quad \text{for }r\text{ even}.
\end{align*}
Note also that 
\begin{equation*}
D_2=\begin{cases}\mathscr{P}_1 & \text{ on }E'_{k-a,\d',j'}\wedge E_{a,\d,j},\\
\mathscr{P}_2 & \text{ on }E'_{k-a+1,d',j'}\wedge E_{a-1,\d,j} .\end{cases}
\end{equation*}
Thus, $D_{2r,k}=D_1+D_2-D_3$, where $D_3=\left(\begin{array}{cc}\mathscr{P}_1&0\\
0&\mathscr{P}_2\end{array}\right)$.

For the statement about the leading term of $D_{2r,k}$, note first that 
\begin{equation*}
(s+r)(\d^{(n)}d^{(n)})^r+(s-r)(d^{(n)}\d^{(n)})^r \text{ is}
\end{equation*}
\begin{align*}
&(s+r)(\d'd'+\d d)^r \text{ on }E'_{k-a,\d',j'}\wedge E_{a,\d,j},\\
&(s-r)(d'\d'+d\d)^r \text{ on }E'_{k-a+1,d',j'}\wedge E_{a-1,d,j},\text{ and}\\
&(s+r)\left(\begin{array}{cc}\d'd'&(-1)^{k-a+1}\d'd\\(-1)^{k-a+1}d'\d&\d d\end{array}\right)^r\\
&\qquad+
(s-r)\left(\begin{array}{cc}d\d&(-1)^{k-a+1}\d'd\\(-1)^{k-a+1}d'\d&d'\d'\end{array}\right)^r \\
&\qquad\text{ on }\left[\begin{array}{c}E'_{k-a,\d',j'}\wedge E_{a,d,j}\\
E'_{k-a+1,d',j'}\wedge E_{a-1,\d,j}\end{array}\right] .
\end{align*}
On the other hand, the leading term of $D_1$ is
\begin{align*}
&(-1)^r(s+r)(\d'd'+\d d)^r \text{ on }E'_{k-a,\d',j'}\wedge E_{a,\d,j},\\
&(-1)^r(s+r)(\d'd')^r \text{ on }E'_{k-a,\d',j'}\wedge E_{a,d,j},\\
&(-1)^r(s-r)(d'\d')^r \text{ on }E'_{k-a+1,d',j'}\wedge E_{a-1,\d,j},\\
&(-1)^r(s-r)(d'\d'+d\d)^r \text{ on }E'_{k-a+1,d',j'}\wedge E_{a-1,d,j},\\
\end{align*}
the leading term of $D_2$ is
\begin{align*}
&(-1)^r(s+r)(\d'd')^r \text{ on }E'_{k-a,\d',j'}\wedge E_{a,\d,j},\\
&(-1)^r(s-r)(d'\d')^r \text{ on }E'_{k-a+1,d',j'}\wedge E_{a-1,d,j},\\
&(-1)^r\left(\begin{array}{cc}\d'd+d\d&0\\0&d'\d'+\d d\end{array}\right)^{r-1}\\
&\cdot\left(\begin{array}{cc}(s+r)\d'd'+(s-r)d\d&(-1)^{k-a+1}2r\d'd\\
(-1)^{k-a+1}2rd'\d&(s-r)d'\d'+(s+r)\d d\end{array}\right)\\
&\quad \text{ on }\left[\begin{array}{c}E'_{k-a,\d',j'}\wedge E_{a,d,j}\\
E'_{k-a+1,d',j'}\wedge E_{a-1,\d,j}\end{array}\right],
\end{align*}
and the leading term of $D_3$ is
\begin{align*}
&(-1)^r(s+r)(\d'd')^r \text{ on }E'_{k-a,\d',j'}\wedge E_{a,\d,j}\oplus E'_{k-a,\d',j'}\wedge E_{a,d,j},\\
&(-1)^r(s-r)(d'\d')^r \text{ on }E'_{k-a+1,d',j'}\wedge E_{a-1,\d,j}\oplus E'_{k-a+1,d',j'}\wedge E_{a-1,d,j}.
\end{align*}
Since $(-1)^r$ times the leading term of $(s+r)(\d^{(n)}d^{(n)})^r+(s-r)(d^{(n)}\d^{(n)})^r$ is the same as that of $D_{2r,k}$ on $\left[\begin{array}{c}E'_{k-a,\d',j'}\wedge E_{a,d,j}\\
E'_{k-a+1,d',j'}\wedge E_{a-1,\d,j}\end{array}\right]$ by induction, the statement about the leading term follows.
\end{proof}
\begin{remark}
It can be readily checked that in the case $S^1\times S^{n-1}$ ($p=2,q=n$), the operators in the theorem agree with Branson's operators in \cite{Branson:87} (p.256). 
Also, if $k=0$, these operators on functions agree with the operators obtained earlier in \cite{Hong:11}.

\end{remark}


\vspace{1cm}
\noindent Doojin Hong\\
Department of Mathematics\\
University of North Dakota\\
Grand Forks, ND 58202, USA\\
Email: doojin.hong@und.edu
\end{document}